\documentclass[10pt]{amsart}
\usepackage{amssymb}
\usepackage{epsfig}
\usepackage{url}
\theoremstyle{plain}

\newtheorem{thm}{Theorem}[section]
\newtheorem{cor}[thm]{Corollary}

\newtheorem{rem}[thm]{Remark}
\newtheorem{ques}[thm]{Question}
\newtheorem{conj}[thm]{Conjecture}
\newtheorem{exam}[thm]{Example}
\numberwithin{equation}{section}

\def\cal{\mathcal}
\def\bbb{\mathbb}
\def\op{\operatorname}
\renewcommand{\phi}{\varphi}

\newcommand{\Z}{\bbb{Z}}
\newcommand{\Q}{\bbb{Q}}

\newcommand{\C}{\bbb{C}}

\begin{document}
\title{Rational points on certain elliptic surfaces}
\author{Maciej Ulas}

\subjclass[2000]{Primary 11D25, 11D41; Secondary 11G052}

\keywords{Diophantine equations, elliptic surfaces}
\date{}

\begin{abstract}
Let $\mathcal{E}_{f}:y^2=x^3+f(t)x$, where $f\in\Q[t]\setminus\Q$,
and let us assume that $\op{deg}f\leq 4$. In this paper we prove
that if $\op{deg}f\leq 3$, then there exists a rational base
change $t\mapsto\phi(t)$ such that there is a non-torsion section
on the surface $\cal{E}_{f\circ\phi}$. A similar theorem is valid
in case when $\op{deg}f=4$ and there exists $t_{0}\in\Q$ such that
infinitely many rational points lie on the curve
$E_{t_{0}}:y^2=x^3+f(t_{0})x$. In particular, we prove that if
$\op{deg}f=4$ and $f$ is not an even polynomial, then there is a
rational point on $\cal{E}_{f}$. Next, we consider a surface
$\cal{E}^{g}:y^2=x^3+g(t)$, where $g\in\Q[t]$ is a monic
polynomial of degree six. We prove  that if the polynomial $g$ is
not even, there is a rational base change $t\mapsto\psi(t)$ such
that on the surface $\cal{E}^{g\circ\psi}$ there is a non-torsion
section. Furthermore, if there exists $t_{0}\in\Q$ such that on
the curve $E^{t_{0}}:y^2=x^3+g(t_{0})$ there are infinitely many
rational points, then the set of these $t_{0}$ is infinite. We
also present some results concerning diophantine equation of the
form $x^2-y^3-g(z)=t$, where $t$ is a variable.
\end{abstract}

\maketitle

\bigskip

\section{Introduction}

Let $\cal{E}$ be an elliptic surface given by the equation
\begin{equation*}
\cal{E}: y^2z=x^3+A(t)xz^2+B(t)z^3,
\end{equation*}
where $A,\;B\in\Q[t]$. The discriminant for $\cal{E}$ is defined
by $\Delta(t)=-16(4A(t)^3+27B(t)^2)$, while $j$-invariant is
defined by $j(t)=-1728(4A(t))^3/\Delta(t)$. We call the surface
$\cal{E}$ isotrivial if its $j$-invariant is constant. We say that
the surface $\cal{E}$ splits if there exists an elliptic curve $C$
such that $\cal{E} \simeq C \times \bbb{P}$ over $\Q$. In the
sequel by an elliptic surface we mean a non-split one. There is a
natural projection on $\cal{E}$ given by $\pi:\cal{E}\ni
([x:y:z],\;t)\mapsto t\in \bbb{P}.$ The mapping $\sigma:\bbb{P}
\rightarrow \cal{E}$ fulfilling a condition
$\pi\circ\sigma=id_{\bbb{P}}$ will be called a section on
$\cal{E}$. Throughout the paper, by a section we mean one defined
over $\Q$. Let us note that there is always zero section on
$\cal{E}$ given by $\sigma_{0}=([0:1:0],\;t)$. We can look on the
surface $\cal{E}$ as on an elliptic curve defined over $\Q(t)$.
Therefore, we have the Mordell-Weil type theorem for $\cal{E}$,
which says that the set of sections (or equivalently points on
$\cal{E}$ defined over $\Q(t)$) forms a finitely generated abelian
group. Because for all but finitely many $t\in\Q$ a fibre
$\cal{E}$ of the mapping $\pi$ is an elliptic curve, a natural
question arises: what can we say about the set of $t\in\Q$ such
that the elliptic curve $\cal{E}_{t}$ has a positive rank? In case
when we have a non-torsion section on $\cal{E}$, this question
follows trivially from Silverman's specialization theorem
(\cite{Sil}, page 368). This theorem says that for all but
finitely many $t\in\Q$ the curve $\cal{E}_{t}$ has a positive
rank. The second interesting question concerns the existence of
rational curves on the surface $\cal{E}$. Let us note that the
existence of a rational curve on $\cal{E}$, say
$(x(u),\;y(u),\;\psi(u)),$ gives us a rational base change
$t=\psi(u)$ such that $\sigma=(x(u),\;y(u))$ is a section on the
surface $\cal{E}_{\psi}: y^2=x^3+A(\psi(u))x+B(\psi(u))$. As we
will see, in many cases $\sigma$ is a non-torsion section. It is
worth noting here that the problem of this kind was considered in
Whitehead's paper (\cite{Whi}). He proved that a rational curve
lies on the surface given by the equation $z^2=f(x,y)$, where
$f\in\Q[x,y]$ and $\op{deg}f=3$. It is easy to see that the
surface defined in this way is birationally equivalent to the
surface $\cal{E}$ for certain $A,\;B\in\Q[t]$ with $\op{deg}A\leq
2,\;\op{deg}B\leq 3.$

Let us also note that the existence of a rational base change
$t=\psi(u)$ such that we have a non-torsion section on
$\cal{E}_{\psi}$, and Silverman's specialization theorem imply
that for all but finitely many $u\in\Q$, each fibre
$\cal{E}_{\psi(u)}$ has a dense set of rational points.

In Section 2 we consider a surface of the form $\cal{E}_{f}:
y^2=x^3+f(t)x$, where $f\in\Q[t]$ and $\op{deg}f\leq 4$. If
$\op{deg}f\leq 3$, then we show that there exists a rational base
change $t=\phi(s)$ such that there is a non-torsion section on the
surface $\cal{E}_{f\circ\phi}$. A similar theorem is proved in
case when $\op{deg}f=4$ and with the assumption that there exists
$t_{0}\in\Q$ such that there are infinitely many rational points
on the curve $E_{t_{0}}: y^2=x^3+f(t_{0})x$. In particular, we
prove that if the polynomial of degree four is not even, then
there is a non-trivial rational point on the surface
$\cal{E}_{f}$.

In Section 3 we consider a surface of the form $\cal{E}^{g}:
y^2=x^3+g(t)$, where $g\in\Q[t]$ is a monic polynomial of degree
six. In this case we prove that if $g$ is not an even polynomial,
then there is a rational base change $t=\chi(u)$ such that there
exists a non-torsion section on $\cal{E}^{g\circ\chi}$. Moreover,
in case when the polynomial $g$ is even, and there exists
$t_{0}\in\Q$ such that the curve $E^{t_{0}}: y^2=x^3+g(t_{0})$
contains infinitely many rational points, then the set of
$t_{0}\in\Q$ such that $E^{t_{0}}$ has a positive rank is
infinite.

In Section 4 we present some results concerning diophantine
equation of the form
\begin{equation*}
x^2-y^3-g(z)=t,
\end{equation*}
where $g(z)=z^6+az^4+bz^3+cz^2+dz+e\in\Z[z]$ and $t$ is a
variable. We will deal with the solution of this equation in the
ring of polynomials $\Q[t]$. In particular, we prove that if
$a\equiv 1 \pmod 2$ and $b\neq 0$, then the above equation has
infinitely many solutions in $\Q[t]$.

In Section 5 we present some results about rational points on
certain non-isotrivial elliptic surfaces.

\section{Rational points on $\cal{E}_{f}:y^2=x^3+f(t)x$}

Let $f\in\Q[t]\setminus\Q$ and let us assume that $\op{deg}f\leq
4$ and $f$ has at least two different complex roots. For such $f$
we consider a surface $\cal{E}_{f}$ given by the equation
\begin{equation*}
\cal{E}_{f}: y^2=x^3+f(t)x.
\end{equation*}
Because $f$ does not have a root of multiplicity four, our
elliptic surface $\cal{E}_{f}$ is non-split. For a given  $t\in\Q$
let us denote the curve $y^2=x^3+f(t)x$ by $E_{t}$. It is worth
noting that for a fixed $t\in\Q$, the torsion part of the group
$E_{t}(\Q)$ is isomorphic to one of the following (\cite{Sil},
page 323): $\Z/4\Z$, if $f(t)=4$;\;$\Z/2\Z\times \Z/2\Z$, if
$-f(t)$ is a square;\; $\Z/2\Z$ if $f(t)$ does not fulfill any of
these conditions. As an immediate consequence we obtain that if
there is a rational base change $t=\beta(u)$ such that we have a
section $\sigma=(x,\;y)$ with $y\neq 0$ on the surface
$\cal{E}_{f\circ \beta}$, then $\sigma$ is a non-torsion section.

We show the following

\pagebreak

\begin{thm}\label{thm1}
\noindent
\begin{enumerate}
\item If $\op{deg}f\leq 3$, then there exists a rational base
change $t=\phi(s)$ such that there is a non-torsion section on the
surface $\cal{E}_{f\circ\phi}$.

\item If $\op{deg}f=4$ and there is $t_{0}\in\Q$ such that the
curve $E_{t_{0}}$ has infinitely many rational points, then there
exists a rational base change $t=\psi(r)$ such that there is a
non-torsion section on the surface $\cal{E}_{f\circ\psi}$.
\end{enumerate}
\end{thm}
\begin{proof}
For the proof of our theorem it will be convenient to work with
the surface $\cal{E}'_{f}$ given by the equation
\begin{equation*}
\cal{E}'_{f}: XY^2=X^2+f(t),
\end{equation*}
which is birationally equivalent to $\cal{E}_{f}$ by the mapping
$(x,y,t)=(X,XY,t)$ with inverse $(X,Y,t)=(x,y/x,t)$. Let us denote
$F(X,\;Y,\;t):=XY^2-X^2-f(t)$.

\bigskip

Proof of part (1).

Let $f\in\Q[t]$ and $\op{deg}f\leq 3$. Without loss of generality
we can assume that $f(t)=at^3+bt^2+ct+d$ for some
$a,\;b,\;c,\;d\in\Z$ with $a\neq 0$ or $b\neq 0$. If $a=b=0$, then
$f$ has degree 1 and if we put $t=(s^4-d)/c$, the surface splits
over $\Q(s)$. Let us put $X=pT+q,\;Y=rT+s,\;t=T$. For $X,\;Y,\;t$
defined in this way we obtain
\begin{equation*}
F(X,\;Y,\;t)=a_{0}+a_{1}T+a_{2}T^2+a_{3}T^3,
\end{equation*}
where
\begin{align*} a_{0}=&-d-q^2+qs^2,\quad a_{1}=-c-2pq+2qrs+ps^2,\\
                     & a_{2}=-b-p^2+qr^2+2prs,\quad a_{3}=-a+pr^2.
\end{align*}
Let us note that the system of equations $a_{2}=a_{3}=0$ has
exactly one solution given by
\begin{equation}\label{R1}
p=\frac{a}{r^2},\quad q=\frac{a^2+br^4-2ar^3s}{r^6}.
\end{equation}
For $p,\;q$ given in this way the equation $F(pT+q,\;rT+s,\;T)=0$
has the root $T=-\phi_{1}(r,\;s)/\phi_{2}(r,\;s)$, where
\begin{align*}
&\phi_{1}(r,\;s)=a^4+2a^2br^4+b^2r^8+dr^{12}-4ar^3(a^2s+br^4)s+r^6(3a^2-br^{4})s^2+2ar^9s^3,\\
&\phi_{2}(r,\;s)=r^4(2a^3+2abr^4+cr^8-2r^3(3a^2+br^4)s+3ar^6s^2).
\end{align*}

We have obtained a two-parametric solution of the equation
defining the surface $\cal{E}'_{f}$. Let us define
$\phi(s):=-\phi_{1}(1,\;s)/\phi_{2}(1,\;s)$ and put $t=\phi(s)$.
We can see that we have a section $\sigma
=(p\phi(s)+q,\;(p\phi(s)+q)(\phi(s)+s))$ on the surface
$\cal{E}_{f\circ\phi}$. Because $(p\phi(s)+q)(\phi(s)+s)$ is a
nonzero rational function, the section $\sigma$ is not of order
two, which proves that it is a non-torsion section.

\bigskip

Proof of part (2).

Because $\op{deg}f=4$, we can assume without loss of generality
that $f(t)=at^4+bt^2+ct+d$ for certain $a,\;b,\;c,\;d\in\Z$ with
$a\neq 0$. From the assumption, there exists $t_{0}\in\Q$ such
that $(x_{0},\;y_{0},\;t_{0})$ is a rational point on
$\cal{E}_{f}$ and $x_{0}\neq 0$. Then the point
$(x_{0},\;y_{0}/x_{0},\;t_{0})$ is a rational point on the surface
$\cal{E}'_{f}$.

Let us now put $X=pT^2+qT+x_{0},\;Y=rT+y_{0}/x_{0},\;t=T+t_{0}$.
For $X,\;Y,\;t$ defined in this way we get
\begin{equation*}
F(X,\;Y,\;t)=(a_{1}T+a_{2}T^2+a_{3}T^3+a_{4}T^4)/x_{0}^2,
\end{equation*}
where
\begin{align*}
&a_{1}=x_{0}^2(c+2bt_{0}+4at_{0}^3-2ry_{0})+q(2x_{0}^3-y_{0}^2),\\
&a_{2}=x_{0}(bx_{0}+q^2x_{0}+6at_{0}^2x_{0}-r^2x_{0}^2-2qry_{0})+p(2x_{0}^3-y_{0}^2),\\
&a_{3}=x_{0}(2pqx_{0}-qr^2x_{0}+4at_{0}x_{0}-2pry_{0}),\\
&a_{4}=(a+p^2-pr^2)x_{0}^2.
\end{align*}
If now $2x_{0}^3-y_{0}^2\neq 0$, then the system of equations
$a_{1}=a_{2}=0$ is triangular with respect to $p,\;q$. Because the
curve  $E_{t_{0}}$ has infinitely many rational points, by fixed
$a,\;b,\;c,\;d$ we will choose $x_{0},\;y_{0}$ such that
$2x_{0}^3-y_{0}^2\neq 0$, and the system $a_{1}=a_{2}=0$ has a
solution fulfilling the condition $p\neq 0$ or $q\neq 0$.
Therefore, we obtain
\begin{equation}\label{R2}
q=-\frac{x_{0}^2(c+2bt_{0}+4at_{0}^3-2ry_{0})}{2x_{0}^3-y_{0}^2},\;
p=-\frac{x_{0}(bx_{0}+q^2x_{0}+6at_{0}^2x_{0}-r^2x_{0}^2-2qry_{0})}{2x_{0}^3-y_{0}^2}.
\end{equation}
For $p,\;q$ defined in this way, the equation $F(pT^2+qT+x_{0},\;
rT+y_{0}/x_{0},\;T+t_{0})=0$ has the triple root $T=0$ and the
root
\begin{equation}\label{R3}
T=-\frac{2pqx_{0}-qr^2x_{0}+4at_{0}x_{0}-2pry_{0}}{(a+p^2-pr^2)x_{0}}=:\psi(r)-t_{0}.
\end{equation}
If we now put $t=\psi(r)$, then on the surface
$\cal{E}_{f\circ\psi}$ there is a section
$\sigma=(pT^2+qT+x_{0},\;(rT+y_{0}/x_{0})(pT^2+qT+x_{0}),\;r)$,
where $p,\;q$ are given by (\ref{R2}) and $T$ is given by
(\ref{R3}). Since $(rT+y_{0}/x_{0})(pT^2+qT+x_{0})\neq 0$, the
section $\sigma$ is not of order two, which proves that it is a
non-torsion section.
\end{proof}

Here a natural and nontrivial question arises concerning the
construction of polynomials $f$ of degree four for which there is
a rational point with $y\neq 0$ on the surface $\cal{E}_{f}$. It
turns out that there exists a wide class of polynomials with this
property.

Now we will show the following

\begin{thm}\label{thm2}
If $f\in\Q[t]$, $\op{deg}f=4$ and $f(t)\neq f(-t)$, then there
exists a rational base change $t= \varphi(u)$ such that on the
surface $\mathcal{E}_{f\circ\varphi}$ we have a non-torsion
section.
\end{thm}

\begin{proof}
Without loss of generality we can assume that
$f(t)=at^4+bt^2+ct+d$ for certain $a,\;b,\;c,\;d\in\Z$, where
$ac\neq 0$. Let $u$ be a variable and let us put $x=au^2$ and
treat our surface as a curve of degree 4 defined over $\Q(u)$,
i.e. we consider the curve
\begin{equation*}
C_{1}:\quad  y^2=a^2u^2t^4+abu^2t^2+acu^2t+adu^2+a^3u^6=:h_{1}(t).
\end{equation*}
Let us note that the point at infinity on the curve $C_{1}$ is
rational. Let us now put $t=T,\;y=auT^2+pT+q$. Then
\begin{equation*}
(auT^2+pT+q)^2-h_{1}(T)=a_{0}+a_{1}T+a_{2}T^2+a_{3}T^3,
\end{equation*}
where
\begin{align*}
&a_{0}=-q^2+adu^2+a^3u^6,\quad a_{1}=-2pq+acu^2,\\
&a_{2}=-p^2-2aqu+abu^2,\quad a_{3}=-2apu.
\end{align*}
The system of equations $a_{2}=a_{3}=0$ has a solution given by
$p=0,\;q=bu/2$. For $p,\;q$ defined in this way, the equation
$(auT^2+pT+q)^2-h_{1}(T)=0$ has the root
\begin{equation*}
T=-\frac{-b^2+4ad+4a^3u^4}{4ac}=:\varphi(u).
\end{equation*}
We have shown that with the assumption $ac\neq 0$ on the surface
$\cal{E}_{f\circ\varphi}$ there is a section
\begin{equation*}
\sigma_{1}
=\Big(au^2,\;\frac{(-b^4-8abc^2+8ab^2d-16a^2d^2)u+8a^3(b^2-4ad)u^5-16a^6u^9}{16ac^2}\Big),
\end{equation*}
which is clearly non-torsion.
\end{proof}

From the above theorem we obtain two interesting corollaries

\begin{cor}\label{cor3}
If $f(t)=at^4+bt^2+ct+d\in\Z[t]$, $a,\;c\in\{-1,\;1\}$ and
$b\equiv 0 \pmod 2$, then the diophantine equation $y^2=x^3+f(t)x$
has infinitely many solutions in integers.
\end{cor}

\begin{cor}\label{cor4}
If $f\in\Q[t]$, $\op{deg}f=4$, $f$ is not an even polynomial and
$f$ has at least two complex roots, then the diophantine equation
$v^2=u^4+f(w)$ has infinitely many rational parametric solutions.
\end{cor}
\begin{proof}
Let us denote $S: v^2=u^4+f(w)$. Using the method described in
(\cite{Mor}, page 77) we obtain that $S$ is birationally
equivalent with the surface
\begin{equation*}
\cal{E}:y^2=x^3-4f(t)x.
\end{equation*}
The mapping from $\mathcal{S}$ to $\mathcal{E}$ is given by
\begin{equation*}
(u,\;v,\;w)=\Big(\frac{y}{2x},\;\frac{2x^3+y^2}{4x^2},\;t\Big),
\end{equation*}
while the inverse mapping is of the form
\begin{equation*}
(x,\;y,\;t)=(-2(u^2-v),\;-4u(u^2-v),\;w).
\end{equation*}
Applying now Theorem \ref{thm2} we obtain the statement of our
corollary.
\end{proof}

An interesting question arises here concerning the existence of a
non-trivial rational point on the surface $\cal{E}_{f}$ in case
when $f(t)=at^4+bt^2+d$ for some $a,\;b,\;d\in\Z.$ It is worth
noting that if the equation $f(t)=0$ has a rational root $t_{0}$,
then on the surface $\cal{E}_{f}$ we have a rational curve
$(x,\;y,\;t)=(u^2,\;u^3,\;t_{0})$ and we can use the second part
of Theorem \ref{thm1} to construct other rational curves on
$\cal{E}_{f}$. Without any difficulty we can give other infinite
families of polynomials fulfilling the conditions of the second
part of Theorem \ref{thm1}. For instance, if
$f(t)=at^4+bt^2+u(v^2-u)$ , then on the curve $E_{0}:
y^2=x^3+u(v^2-u)x$ there is a point $(u,\;uv)$ which is not of
finite order if $uv\neq 0$.

One can check with computer that if
$\op{max}\{|a|,\;|b|,\;|d|\}\leq 100$, then there exists $t\in\Q$
such that infinitely many rational points lie on the curve $E_{t}:
y^2=x^3+f(t)x$. This leads us to the following

\begin{conj}\label{conj1}
Let $a,\;b,\;d\in\Z$ and $f(t)=at^4+bt^2+d$. Then there exists
$t_{0}\in\Q$ such that there are infinitely many rational points
on the curve $E_{t_{0}}$.
\end{conj}

\section{Rational points on $\cal{E}^{g}:y^2=x^3+g(t)$}

Let $g\in\Q[t]$ be a monic polynomial of degree 6 and let
$g(t)\neq t^6$. For such $g$ let us consider the surface
\begin{equation*}
\cal{E}^{g}: y^2=x^3+g(t).
\end{equation*}
For a given $t\in\Q$ let us denote the curve $y^2=x^3+g(t)$ by
$E^{t}$. Let us recall how a torsion part of the curve $E^{t}$
looks like with a fixed $t\in\mathbb{Q}$  (\cite{Sil}, page 323).
If $g(t)=1$, then $\op{Tors} E^{t}\cong\mathbb{Z}/6\mathbb{Z}$. If
$g(t)\neq 1$ and $g(t)$ is a square in $\Q$, then $\op{Tors}
E^{t}=\{\mathcal{O},\;(0,\;\sqrt{g(t)}),\;(0,\;-\sqrt{g(t)})\}$.
In case when $g(t)=-432$ we have
$\op{Tors}E^{t}=\{\mathcal{O},\;(12,\;36),\;(12,\;-36)\}$. If
$g(t)\neq 1$ and $g(t)$ is a cube in $\Q$, then $\op{Tors}
E^{t}=\{\mathcal{O},\;(-\sqrt[3]{g(t)},\;0)\}$. In the remaining
cases we have $\op{Tors} E^{t}=\{\mathcal{O}\}$. As an immediate
consequence we obtain that if there is a rational base change
$t\mapsto \beta(t)$ such that on the curve $\cal{E}^{g\circ
\beta}$ we have the section $\sigma=(x,\;y)$ with $xy\neq 0$, then
$\sigma$ is a non-torsion section.

We show the following

\begin{thm}\label{thm5}
Let $g\in\Q[t]$ be a monic polynomial of degree six. If $g$ is not
an even polynomial, then there exists a rational base change
$t=\chi(u)$ such that there is a non-torsion section on the curve
$\cal{E}^{g\circ\chi}$.
\end{thm}
\begin{proof}
Without loss of generality we can assume that
$g(t)=t^6+at^4+bt^3+ct^2+dt+e$ for certain $a,b,c,d,e\in\Z$ with
$b\neq 0$ or $d\neq 0$. Let now $C_{2}$ denote a curve defined
over the field $\Q(t)$ obtained from $\cal{E}^{g}$ after
substituting $x=\frac{u^2-a}{3}-t^2$. We consider the curve of the
form
\begin{align*}
C_{2}:
y^2=u^2t^4+bt^3&-\frac{a^2-3c-2au^2+u^4}{3}t^2+dt\\
               &+\frac{-a^3+27e+3a^2u^2-3au^4+u^6}{27}=:h_{2}(t).
\end{align*}
Note that the point at infinity on the curve $C_{2}$ is rational.
Let us put $t=T,\;y=uT^2+pT+q$. Then
\begin{equation*}
(uT^2+pT+q)^2-h_{2}(T)=a_{0}+a_{1}T+a_{2}T^2+a_{3}T^3,
\end{equation*}
where
\begin{align*}
&a_{0}=\frac{a^3-27e+27q^2-3a^2u^2+3au^4-u^6}{27},\quad
a_{1}=-d+2pq,\\
\quad &a_{2}=\frac{a^2-3c+3p^2+6qu-2au^2+u^4}{3},\quad a_{3}=
-b+2pu.
\end{align*}
Solving the system of equations $a_{2}=a_{3}=0$ with respect to
$p,\;q$ we obtain
\begin{equation}\label{R4}
p=\frac{b}{2u},\;q=\frac{-3b^2-4a^2u^2+12cu^2+8au^4-4u^6}{24u^3}.
\end{equation}
 Now, if $p,\;q$ are given by (\ref{R4}), then the equation $(uT^2+pT+q)^2-h_{2}(T)=0$
 has a root $T=-\chi_{1}(u)/\chi_{2}(u)=:\chi(u)$, where
 \begin{align*}
\chi_{1}(u)=-27b^4&-72b^2(a^2-3c)u^2-48(a^4-3ab^2-6a^2c+9c^2)u^4\\
                  &+8(16a^3-9b^2-72ac+216e)u^6-96(a^2-3c)u^8+16u^{12},
\end{align*}
\begin{equation*}
\hskip -2.4cm
\chi_{2}(u)=72u^2(3b^3+4b(a^2-3c)u^2-8(ab-3d)u^4+4bu^6).
\end{equation*}
Our computations imply that on the curve $\cal{E}^{g\circ\chi}$ we
have the section $\sigma_{2} =((u^2-a-3T^2)/3,\;uT^2+pT+q)$, where
$p,\;q$ are given by (\ref{R4}) and $T=\chi(u)$. It is easy to see
that the section $\sigma_{2}$ is non-torsion.

It is also worth noting that the assumption $b\neq 0$ or $d\neq 0$
is essential for the employed method because in the opposite case
the function $\chi_{2}$ is identically equal to zero.
\end{proof}

Here a natural question arises whether the assumption that for a
certain $t_{0}\in\Q$ infinitely many rational points lie on the
curve $E^{t_{0}}$ enables to construct a rational curve on the
surface $\cal{E}^{g}$. Unfortunately, we are not able to show such
a theorem with any even polynomial $g$. However, we can prove the
following

\bigskip

\begin{thm}\label{thm6}
Let $g\in\Q[t]$ be a monic and even polynomial of degree six. If
there exists $t_{0}\in\Q$ such that there are infinitely many
rational points on the curve $E^{t_{0}}$, then the set of $t\in\Q$
such that $E^{t}$ has positive rank is infinite.
\end{thm}
\begin{proof}
Because $g$ is even we can assume that $g(t)=t^6+at^4+ct^2+e$ for
certain $a,\;c,\;e\in\Z$ with $a\neq 0$ or $c\neq 0$. The case
$a=c=0$ will be discussed in the next section. For the proof it
will be more convenient to work with the surface $\cal{F}^{g}$
given by the equation
\begin{equation*}
\cal{F}^{g}: Y^2+2t^3Y=X^3+at^4+ct^2+e.
\end{equation*}
Let us denote $G(X,\;Y,\;t):=Y^2+2t^3Y-(X^3+at^4+ct^2+e)$. Then
$\cal{E}^{g}$ is birationally equivalent with $\mathcal{F}^{g}$ by
the mapping $(x,\;y,\;t)=(X,\;Y+t^3,\;t)$ with the inverse
$(X,\;Y,\;t)=(x,\;y-t^3,\;t)$. From the assumption there exists
$t_{0}\in\Q$ such that there are infinitely many rational points
on $E^{t_{0}}$. Thus, there is a rational point
 $(x_{0},\;y_{0},\;t_{0})$ on the surface $\mathcal{E}^{g}$ such
 that $x_{0}y_{0}\neq 0$. Then the point
$(x_{0},\;y_{0}-t_{0}^3,\;t_{0})$ is on the surface
$\mathcal{F}^{g}$. Let us put
$X=pT+x_{0},\;Y=qT+y_{0}-t_{0}^3,\;t=T+t_{0}.$ Then
\begin{equation*}
G(X,\;Y,\;t)=a_{1}T+a_{2}T^2+a_{3}T^3+a_{4}T^4,
\end{equation*}
where
\begin{align*}
&a_{1}=-3px_{0}^2+2qy_{0}-2ct_{0}-4at_{0}^3-6t_{0}^5+6t_{0}^2y_{0},\\
&a_{2}=q^2+6qt_{0}^2-3p^2x_{0}-c+6at_{0}^2-6t_{0}^4+6t_{0}y_{0},\\
&a_{3}=-p^3+6qt_{0}-4at_{0}-2t_{0}^3+2y_{0},\\
& a_{4}=2q-a.
\end{align*}
Solving the system of equations $a_{1}=a_{4}=0$ with respect to
$p,\;q$ we obtain
\begin{equation}\label{R5}
p=-\frac{2ct_{0}+4at_{0}^3+6t_{0}^5-ay_{0}-6t_{0}^2y_{0}}{3x_{0}^2},\quad
q=\frac{a}{2}.
\end{equation}
For $p,\;q$ defined in this way, the equation
$G(pT+x_{0},\;qT+y_{0}-t_{0}^3,\;T+t_{0})=0$ has the root $T=0$
and the root
\begin{equation}\label{R6}
T=-\frac{q^2+6qt_{0}^2-3p^2x_{0}-c+6at_{0}^2-6t_{0}^4+6t_{0}y_{0}}{-p^3+6qt_{0}-4at_{0}-2t_{0}^3+2y_{0}}.
\end{equation}
From the above computations we can see that the point
$(pT+x_{0},\;qT+y_{0}-t_{0}^3,\;T+t_{0})$ for $p,\;q$ given by
(\ref{R5}) and $T$ defined by (\ref{R6}), lies on the surface
$\mathcal{F}^{g}$. Hence we obtain the point
$P=(pT+x_{0},\;qT+y_{0}-t_{0}^3+(T+t_{0})^3,\;T+t_{0})$ on the
surface $\mathcal{E}^{g}$. Because the set of rational points on
$E^{t_{0}}$ is infinite, we can assume that the coordinates of the
point $P$ are nonzero, $g(T)\neq 0,-432$ and $g(T+t_{0})/g(t_{0})$
is not a sixth power. If we now put $t_{1}=T+t_{0}$, then the
curve $E^{t_{1}}$ has infinitely many rational points.

Let us now suppose that we have already constructed
$t_{1},\;\ldots,\;t_{n}$ such that the curve $E^{t_{i}}$ has a
positive rank for $i=1,\;\ldots,\;n.$ Then we can apply the above
procedure to the point $(x_{n},\;y_{n},\;t_{n})$, where
$(x_{n},\;y_{n})$ is a rational point on $E^{t_{n}}$ such that $T$
given by (\ref{R6}) fulfills the conditions: $g(T+t_{n})\neq
0,\;-432$ and $g(T+t_{n})/g(t_{i})$ is not a sixth power for
$i=1,\;\ldots, \;n$. Why can we find such $T$ given by (\ref{R6})
and fulfilling these conditions? Let us note that if the
polynomial $g$ does not have the root of multiplicity 5, then
there are only finitely many rational points on every curve
$g(u)=g(t_{i})v^6$ (of genus $> 1$) for $i=1,\;\ldots,\;n.$ It is
an immediate consequence of the Faltings theorem (\cite{Fal}). The
case when $g$ has the root of multiplicity 5 (it is a rational
root then) can be easily excluded, as then the surface
$\cal{E}^{g}$ is rational over $\Q$. Because there are infinitely
many rational points on the curve $E^{t_{n}}$, we can see that $T$
given by (\ref{R6}) can be selected to fulfill all the necessary
conditions. Using now the previous reasoning we can construct an
infinite set of values $t\in\Q$ such that $E^{t}$ has a positive
rank.
\end{proof}

\begin{rem}\label{rem7}
{\rm Let us note  that if $g(t_{0})=0$ for a rational number
$t_{0}$, then the set of rational points on the curve
$E^{t_{0}}:\;y^2=x^3$ is parametrized by $x=u^2,\;y=u^3$. Using
the reasoning from Theorem \ref{thm6} we can easily deduce that in
this case it is possible to construct a rational curve on the
surface $\cal{E}^{g}$.}
\end{rem}

From the above remark we obtain the following

\begin{cor}\label{cor8}
Let $h\in\Q[t]$ with $\op{deg}h=5,\;h(0)=1$ and let us consider
the surface $\cal{S}:\;y^2=x^3+h(t)$. Then, there is a rational
base change $t=\gamma(u)$ such that we have a non-torsion section
on the surface $\cal{S}^{\gamma}:\;y^2=x^3+h(\gamma(u))$.
\end{cor}
\begin{proof}
Let us note that the surface $\cal{S}$ is birationally  equivalent
with the surface $\cal{E}^{g}$, where $g(t)=t^6h(1/t)$. The
mapping from $\cal{S}$ to $\cal{E}^{g}$ is given by
$(x,\;y,\;t)\mapsto (x/t^2,\;y/t^3,\;1/t)$. Because $g(0)=0$, we
can use the Remark \ref{rem7} to obtain the statement of our
corollary.
\end{proof}

\begin{exam}\label{exam9}
{\rm Let $g(t)=t^6+t^2+1$ and let us consider the surface
$\cal{E}^{g}: y^2=x^3+g(t)$. For $t_{0}=1$ on the curve
$E^{1}:y^2=x^3+3$ we have a non-torsion point $P=(1,\;2)$. Now we
calculate the quantities $p,\;q$ given by (\ref{R5}) and $T$ given
by (\ref{R6}) from the proof of Theorem \ref{thm6}. We obtain
$p=16/13,\;q=-1/13,\;T=-358/169$ and next
$t_{1}=T+t_{0}=-189/169$. Thus, we can see that on the curve
\begin{equation*}
E^{t_{1}}:\;y^2=x^3+\frac{47*2085456070589}{ 13^{12}}
\end{equation*}
we have a non-torsion point
\begin{equation*}
P=\Big(-\frac{3531}{2197},\;\frac{1137934}{4826809}\Big).
\end{equation*}}
\end{exam}

\bigskip

Similarly to the case of the surface $\cal{E}_{f}$ considered in
Section 2, we can ask whether for a given polynomial $g$ of the
form $g(t)=t^6+at^4+ct^2+e$ there is $t_{0}\in\Q$ such that the
curve $E^{t_{0}}$ has infinitely many rational points.

In the following section we will prove that the answer to this
question is positive for polynomials of the form $g(t)=t^6+e$.
Using computer we checked that if $\op{max}\{|a|,\;|c|,\;|e|\}\leq
10$, then there exists $t\in\Q$ such that infinitely many rational
points lie on the curve $E^{t}: y^2=x^3+t^6+at^4+ct^2+e$. This
leads us to the following

\begin{conj}\label{conj2}
Let $a,\;c,\;e\in\Z$ and $g(t)=t^6+at^4+ct^2+e$. Then there exists
$t_{0}\in\Q$ such that there are infinitely many rational points
on the curve $E^{t_{0}}$.
\end{conj}

In view of Theorem \ref{thm6} a natural question arises

\begin{ques}\label{ques1}
Let $g(t)=t^6+at^4+ct^2+e$ and let us consider the surface
$\cal{E}^{g}$. What are the conditions guaranteeing the existence
of a rational base change $t=\kappa(u)$ such that there is a
non-torsion section on the surface $\cal{E}^{g\circ\kappa}$?
\end{ques}

\section{Some results on the diophantine equation $x^2-y^3-g(z)=t$}

Let $g(z)=z^6+az^4+bz^3+cz^2+dz+e \in\Z[z]$ and let $t$ be a
variable. In this section we will deal with the diophantine
equation of the form
\begin{equation}\label{R7}
x^2-y^3-g(z)=t.
\end{equation}

We will show that if there are infinitely many rational points on
the curve $C: v^2=s^4-12as^2+48bs+6(a^2-12c),$ then the equation
(\ref{R7}) has infinitely many solutions in the ring of
polynomials $\Q[t]$. In case when $g(t)=t^6+e$, we will use this
result to prove the promised theorem concerning the existence of a
rational base change $t=\chi_{1}(s)$ such that there exists a
non-torsion section on the surface $\cal{E}^{g\circ\chi_{1}}$. We
will also present some results concerning the representability of
integers in the form $x^2-y^3-g(z)$.

We start with the following

\begin{thm}\label{thm10}
If there are infinitely many rational points on the curve $C:
v^2=s^4-12as^2+48bs+6(a^2-12c)$, in particular if $b\neq 0$ and
$a\equiv 1\pmod 2$, then the equation {\rm (\ref{R7})} has
infinitely many solutions in the ring of polynomials $\Q[t]$.
\end{thm}
\begin{proof}
Let us denote $G(x,y,z):=x^2-y^3-g(z)$ and observe that the
question about solvability of the equation $G(x,y,z)=t$ in
polynomials with rational coefficients is equivalent to the
question about the construction of polynomials $x,\;y,\;z\in\Q[t]$
such that $\op{deg}G(x(t),y(t),z(t))=1.$ Let us now put
$x=3T^3+pT^2+qT+r,\;y=2T^2+sT+u,\;z=T$. Then
\begin{equation*}
G(3T^3+pT^2+qT+r,\;2T^2+sT+u,\;T)=a_{0}+a_{1}T+a_{2}T^2+a_{3}T^3+a_{4}T^4+a_{5}T^5,
\end{equation*}
where
\begin{align*}
&a_{0}=r^2-u^3-e,\quad a_{1}=-d+2qr-3su^2,\\
&a_{2}=-c+q^2+2pr-3s^2u-6u^2, \quad a_{3}=-b+2pq+6r-s^3-12su,\\
&a_{4}=-a+p^2+6q-6s^2-12u,\quad a_{5}=6(p-2s).
\end{align*}
Solving the system of equations $a_{3}=a_{4}=a_{5}=0$ with respect
to $p,\;q,\;r$ we obtain
\begin{equation}\label{R8}
p=2s,\quad q=\frac{a+2s^2+12u}{6},\quad
r=\frac{3b-2as-s^3+12su}{18}.
\end{equation}
After substituting $p,\;q,\;r$ into the equation $a_{2}=0$ and
solving this equation with respect to $u$, we obtain
\begin{equation}\label{R9}
u=\frac{3s^2+2a \pm \sqrt{s^4-12as^2+48bs+6(a^2-12c)} }{12}.
\end{equation}
Thus, we can see that if infinitely many rational points lie on
the curve
\begin{equation*}
C:\quad v^2=s^4-12as^2+48bs+6(a^2-12c)=:U(s),
\end{equation*}
then all but finitely many points on $C$, by (\ref{R8}) and
(\ref{R9}), give us a triple of polynomials $x,\;y,\;z\in\Q[T]$
such that $G(x(T),y(T),z(T))=a_{1}T+a_{0}$ and $a_{1}\neq 0$.
After substitution $T=(t-a_{0})/a_{1}$ we obtain a solution of the
equation $x^2-y^3-g(z)=t$. Let us also note that we always have
infinitely many rational points on $C$ when the polynomial $U$ has
multiple roots, which is equivalent to the condition
$D:=25a^6-144a^3b^2-2592b^4-180a^4c+5184ab^2c-1296a^2c^2-1728c^3=0$.

Since in the case when $D=0$, the curve $C$ is rational over $\Q$,
we can assume that $D\neq 0$. To show that if $b\neq 0$ and
$a\equiv 1\pmod 2$, then infinitely many rational points lie on
the curve $C$, we transform  $C$ into an elliptic curve with
Weierstrass equation. We can do this because the point at infinity
on the curve $C$ is rational. Using the method described in
\cite{Mor} one more time, we birationally transform the curve $C$
into the curve
\begin{equation*}
E:\;Y^2=X^3-72(a^2-4c)X+64(a^3+36b^2-36ac).
\end{equation*}
The mapping transforming $C$ into $E$ is in the form
\begin{equation*}
(s,\;v)=\Big(\frac{48b-Y}{16a-2X},\;2a+\frac{X}{2}-\Big(\frac{48b-Y}{16a-2X}\Big)^2\Big),
\end{equation*}
while the inverse mapping is given by
\begin{equation*}
(X,\;Y)=(2(-2a+s^2+v),\;4(12b-6as+s^3+sv)).
\end{equation*}

Let us note that we have a rational point $P=(8a,\;48b)$ on the
curve $E$. Using the chord and tangent method of adding points on
elliptic curve we obtain $2P=(x_{1},\;y_{1})$ where
\begin{align*}
&x_{1}=\frac{25a^4-256ab^2+120a^2c+144c}{16b^2},\\
&y_{1}=48b+\frac{(5a^2+12c)(25a^4-384ab^2+120a^2c+144c^2)}{64b^3}.
\end{align*}
Because $a\equiv 1\pmod 2$ the numerator $x_{1}$ is odd, and this
means, that $x_{1}\in\Q\setminus\Z$. From the Nagell-Lutz theorem
(\cite{Sil}, page 77) we know that the torsion points on the
elliptic curve $y^2=x^3+px+q,\;p,\;q\in\Z$ have integer
coordinates, therefore, we see that the point $2P$ is not of
finite order. It proves that the curve $E$ has a positive rank;
and we obtain that there are infinitely many rational points on
the curve $C$.
\end{proof}

\begin{rem}\label{rem11}
{\rm After noticing that the point $P=(8a,\;48b)$ lies on the
curve $E$ from the proof of the above theorem we suspected that
this point is not of finite order for $ab \neq 0$ and any
$c\in\Z$. As Professor Schinzel suggested, it is not true. Indeed,
if $a=6p^2,\;c=p(4b-15p^3)$, then the curve $E$ is elliptic, if
$\Delta=-764411904b^2(3b-16p^3)\neq 0$. In this case the point
$P=(6p^2,\;48b)$ is a point of order three on the curve $E$. If we
now put $p=1,\;b=1,$ then $a=6,\;c=-11$. For $a,\;b,\;c$ defined
in this way, the curve $E$ is birationally equivalent with the
curve $E':\;y^2=x^3-360x+2628$. With the assistance of {\sc APECS}
program (\cite{Co}) we found that the rank of the curve $E'$ is
zero. Despite this, there exists a non-trivial solution of the
equation $x^2-y^3-g(z)=t$ and it turns out that it is valid if
$b\neq 0$; this is equivalent to the fact that the point $P$ is
not of order two. Why is it so? If $b\neq 0$, then the order of
the point $P$ equals at least 3 and $s$-coordinate of preimage of
the point $2P$ (which is different from the point at infinity
$\cal{O}$) equals $(5a^2+12c)/18b$. Because the expression $a_{1}$
from the proof of Theorem \ref{thm10} depends linearly on $d$ and
is not identically equal to zero, then there is at least one
$d\in\Z$, for which $a_{1}=0$ and our method does not give a
solution of the equation $x^2-y^3-g(z)=t$.

It should be noted, however, that there exists a polynomial
$g\in\Z[t]$ for which our method does not give a solution of the
equation $x^2-y^3-g(z)=t$. For example, if
$g(t)=t^6+6t^4+6t^3+9t^2-150t$, then the curve $C$ is birationally
equivalent with the elliptic curve $E': y^2+y=x^3-7$. We have that
$\op{Tors}E'=\{\cal{O},\;(3,4),\;(3,-5)\}$ and using {\sc APECS}
program once again, we find that $E'$ has rank equal to zero. In
this case, our method leads us to the identity
\begin{equation*}
 (3T^3+12T^2+33T+25)^2-(2T^2+6T+10)^3-g(T)=-375.
\end{equation*}
}
\end{rem}

Now we will note several interesting corollaries of Theorem
\ref{thm10}

\begin{cor}\label{cor12}
If  infinitely many rational points lie on the curve
$C:\;v^2=s^4-12as^2+48bs+6(a^2-12c)$, then every polynomial
$h\in\Q[t]$ can be represented in infinitely many ways  in the
form $x^2-y^3-g(z)$, where $x,\;y,\;z\in\Q[t]$.
\end{cor}

In the following corollary we give the promised proof of the
theorem concerning the existence of rational curves on the surface
$y^2=x^3+t^6+e$.

\begin{cor}\label{cor13}
Let $\cal{E}^{g}:\;y^2=x^3+g(t)$, where $g(t)=t^6+e$, then there
exists a rational base change $t=\chi_{1}(s)$ such that there is a
non-torsion section on the surface
$\cal{E}^{g\circ\chi_{1}}:\;y^2=x^3+g(\chi_{1}(s))$
\end{cor}
\begin{proof}
Let us note that if $a=b=c=0$, then the curve $C$ is rational and
the system of equations $a_{2}=a_{3}=a_{4}=a_{5}=0$ from the proof
of Theorem \ref{thm10} has exactly two solutions given by
\begin{align*}
&p_{1}=2s,\quad q_{1}=\frac{2s^2}{3},\quad r_{1}=\frac{s^3}{18},\quad u_{1}=\frac{s^2}{6},\\
&p_{2}=2s,\quad q_{2}=s^2,\quad r_{2}=\frac{s^3}{6},\quad
u_{2}=\frac{s^2}{3}.
\end{align*}
For such $p_{i},\;q_{i},\;r_{i},\;u_{i},\; (i=1,\;2)$ we obtain
the following identities
\begin{equation}\label{R10}
\Big(3T^3+2sT^2+\frac{2s^2}{3}T+\frac{s^3}{18}\Big)^2-\Big(2T^2+sT+\frac{s^2}{6}\Big)^3-(T^6+dT+e)
\end{equation}
\begin{equation*}
\hskip 1cm  =-\frac{648e+s^6}{648}-\frac{648d+6s^5}{648}T,
\end{equation*}
\begin{equation}\label{R11}
\Big(3T^3+2sT^2+2s^2T+\frac{s^3}{6}\Big)^2-\Big(2T^2+sT+\frac{s^2}{3}\Big)^3-(T^6+dT+e)=-\frac{108e+s^6}{108}-dT.
\end{equation}
If now $d=0$, then putting $T=\chi_{1}(s)=-(648e+s^6)/(6s^5)$ the
right side of the identity (\ref{R10}) disappears and on the
surface $\cal{E}^{g\circ \chi_{1}}:\;y^2=x^3+g(\chi_{1}(s))$ we
obtain a section
\begin{equation*}
\sigma =
\Big(\frac{419904e^2-648es^6+s^{12}}{18s^{10}},\;-\frac{272097792e^3-419904e^2s^6+1944es^{12}+s^{18}}{72s^{15}}\Big).
\end{equation*}
It is easy to see that the order of $\sigma$ is not finite.
\end{proof}

\bigskip

 Let us remind that $a_{1}=-d+2qr-3su^2$, where $q,r,s,u$ are
given by (\ref{R8}) and (\ref{R9}) from the proof of the Theorem
\ref{thm10}.

\bigskip

\begin{cor}\label{cor14}
If $d\in\Z$ and on the curve $
C:\;v^2=s^4-12as^2+48bs+6(a^2-12c)$, there is a rational point
such that $a_{1}\neq 0$, then for every integer $n$ the
diophantine equation $x^2-y^3-g(z)=n$ has solutions in rationals
$x,\;y,\;z$ such that there exists an integer $L_{g}$ dependent
only on the polynomial $g$ such that
$L_{g}x,\;L_{g}y,\;L_{g}z\in\Z$. In particular, for $g(z)=z^6$ we
have $L_{g}=24416=2^9*3^5$
\end{cor}
\begin{proof}
In the light of Theorem \ref{thm10} the first part of the
statement is obvious. Now putting $d=e=0,\;s=6$ and next
$T=(n+72)/72$ into the identity  (\ref{R10}) we obtain
\begin{equation*}
\Big(\frac{n^3-72n^2+15552n+373248}{24416}\Big)^2-\Big(\frac{n^2-72n+5184}{2592}\Big)^3-\Big(\frac{n+72}{72}\Big)^6=n.
\end{equation*}
This proves the second part of our corollary.
\end{proof}

\begin{cor}\label{cor15}
Let $g(z)=z^6+dz$. If $d=1$, then for every integer $n$ the
diophantine equation $x^2-y^3-g(z)=n$ has infinitely many
solutions in integers. If $d=-72t^5+1$ for a certain integer $t$,
then for every integer $n$ the diophantine equation
$x^2-y^3-g(z)=n$ has a solution in integers.
\end{cor}
\begin{proof}
Let $n$ be a fixed integer. If $d=-1$, then for the proof we will
use  the identity (\ref{R11}). Let us put $e=0,\;s=6t$ and
$T=-432t^6-n$. For polynomials given by
\begin{align*}
x(t)=\; &3n^3+12t(-1+324t^5)n^2+36t^2(1-288t^5+46656t^{10})n\\
      &+36t^3(-1+432t^5-62208t^{10}+6718464t^{15}),\\
y(t)=\; &2n^2+6t(-1+288t^5)n+12t^2(1-216t^5+31104t^{10}),\\
z(t)=\; &-n-432t^6,
\end{align*}
we get $x(t)^2-y(t)^3-g(z(t))=n$.

If now $d=-72t^5-1$, then we put $e=0,\;s=6t,\;T=-n-72t^6$ into
the identity (\ref{R10}). We obtain that $x^2-y^3-g(z)=n$ for
\begin{align*}
x=\; &3n^3+12t(-1+54t^5)n^2+24t^2(1-72^5+1944t^{10})n\\
   &+12t^3(-1+144t^5-5184t^{10}+93312t^{15}),\\
y=\; &2n^2+6t(-1+48t^5)n+6t^2(1-72t^5+1728t^{10}),\\
z=\; &-n-72t^6.
\end{align*}
\end{proof}

\section{Rational points on some non-isotrivial elliptic surfaces}

In view of our consideration it is natural to ask whether it is
possible to obtain similar results for non-isotrivial elliptic
surfaces of the form
\begin{equation*}
\cal{E}:\;y^2=x^3+A(t)x+B(t),
\end{equation*}
where $A,\;B\in\Q[t]\setminus\{0\}$. If $t\mapsto \alpha (t)$ is a
rational base change, then by $\cal{E}_{\alpha}$ let us denote the
surface $\cal{E}_{\alpha}:\;y^2=x^3+A(\alpha (t))x+B(\alpha(t))$.
Let us also remind that if $\cal{C}:y^2=x^3+m(t)x+n(t)$, where
$m,\;n\in\Z[t]$ is an elliptic curve over $\Q(t)$, then the points
of finite order on $\cal{C}$ have coordinates in $\mathbb{Z}[t]$.

In the following section we will prove the generalization of the
first part of Theorem \ref{thm1} and Theorem \ref{thm2}.

\begin{thm}\label{thm16}
Let $\cal{E}:y^2=x^3+f_{4}(t)x+g_{4}(t)$, where
$f_{4},\;g_{4}\in\Q[t]$. If $\op{deg}f_{4}=3$ and
$\op{deg}g_{4}\leq 4$ or $\op{deg}f_{4}=4,\;\op{deg}g_{4}\leq 4$
and at least one of polynomials $f_{4},\;g_{4}$ is not even, then
there exists rational base change $t=\psi(s)$ such that  we have a
non-torsion section on the surface $\cal{E}_{\psi}$.
\end{thm}
\begin{proof}
Let us denote $H(x,\;y,\;t):=y^2-(x^3+f_{4}(t)x+g_{4}(t))$. Let us
first consider the case when $\op{deg}f_{4}=3$ and
$\op{deg}g_{4}\leq 4$. Without loss of generality we can assume
that $f_{4}(t)=at^3+bt+c,\;g_{4}(t)=dt^4+et^3+ft^2+gt+h$ for some
$a,\;b,\;\ldots,\;h\in\mathbb{Z}$ with $a\neq 0$ and $g_{4}(t)\neq
0$. Let us put $x=pT+q,\;y=rT^2+sT+u,\;t=T$. For $x,\;y,\;t$
defined in this way we obtain
\begin{equation*}
H(x,\;y,\;t)=a_{0}+a_{1}T+a_{2}T^2+a_{3}T^3+a_{4}T^4,
\end{equation*}
where
\begin{align*}
&a_{0}=-h-cq-q^3+u^2,\quad a_{1}=-g-cp-bq-3pq^2+2su,\\
&a_{2}=-f-bp-3p^2q+s^2+2ru,
\;a_{3}=-e-p^3-aq+2rs,\;a_{4}=-d-ap+r^2.
\end{align*}

Let us note that the system of equations $a_{2}=a_{3}=a_{4}=0$ has
a solution given by
\begin{equation}\label{R12}
p=\frac{-d+r^2}{a},\quad
q=-\frac{-d^3+a^3e+3d^2r^2-3dr^4+r^6-2a^3rs}{a^4},
\end{equation}
\begin{equation*}
u=-\frac{-f-bp-3p^2q+s^2}{2r}.
\end{equation*}
If now $p,\;q,\;u$ are given by (\ref{R12}), then the equation
$H(pT+q,\;rT^2+sT+u,\;T)=0$ has a solution
\begin{equation}\label{R13}
T=-\frac{h+cq+q^3-u^2}{g+cp+bq+3pq^2-2su}=:\psi(r,s).
\end{equation}
In this case we obtain a two-parametric solution of the equation
defining the surface $\cal{E}$. For convenience let us put $r=1$
and $\psi(s):=\psi(1,s)$. Therefore, we can see that if
$p,\;q,\;u$ are given by (\ref{R12}) and $T=t=\psi(s)$, then on
the surface $\cal{E}_{\psi}$ we have a section
$\sigma=(pT+q,\;T^2+sT+u)$. Performing affine change of variables
we transform the surface $\cal{E}_{\psi}$ into the
$\cal{E}'_{\psi}:y^2=x^3+f'_{4}(s)x+g'_{4}(s)$, where
$f'_{4},\;g'_{4}\in\Z[s]$. Then $\sigma$ goes to the section
$\sigma'$ on $\mathcal{E}'_{\psi}$. It turns  out that
$x$-coordinate of the section $2\sigma'$ belongs to $\Q(s)
\setminus \Q[s]$. From the remark given at the beginning of this
section we see that $\sigma'$ is not of finite order. We will not
present here the details of this proof as it requires a lot of
computations which are immensely difficult to perform without
computer.

\bigskip

Let us now consider the case when
$\op{deg}f_{4}=4,\;\op{deg}g_{4}\leq 4$ and at least one of the
polynomials $f_{4},\;g_{4}$ is not even. We can assume that
$f_{4}(t)=at^4+bt^2+ct+d,\;g_{4}(t)=et^4+ft^3+gt^2+ht+i$, where
$a,\;b,\;\ldots,\;i\in\Z,\;a\neq 0$ and at least one of the
numbers $c,\;f,\;h$ is non-zero. Let now $C_{3}$ denote a curve
over $\Q(u)$ obtained from $\cal{E}$ after substitution
$x=(u^2-e)/a$. Hence, we consider the curve of the form
\begin{align*}
C_{3}:\;y^2=u^2t^4&+ft^3+\frac{-be+ag-bu^2}{a}t^2+\frac{-ce+ah-cu^2}{a}t\\
                  &+\frac{-a^2de-e^3+a^3i+(a^2d+3e^2)u^2-3eu^4+u^6}{a^3}=:V(t)
\end{align*}
Now putting $y=uT^2+pT+q,\;t=T$ we obtain
\begin{equation*}
(uT^2+pT+q)^2-V(T)=a_{0}+a_{1}T+a_{2}T^2+a_{3}T^3,
\end{equation*}
where
\begin{align*}
&a_{0}=\frac{-a^2de-e^3+a^3i+(a^2d+3e^2)u^2-3eu^4+u^6}{a^3},\; a_{1}=\frac{-ce+ah-2apq+cu^2}{a},\\
&a_{2}=\frac{-be+ag-ap^2-2aqu+bu^2}{a},\; a_{3}=f-2pu.
\end{align*}
The system of equations  $a_{2}=a_{3}=0$ has a solution given by
\begin{equation}\label{R14}
p=\frac{f}{2u},\quad q=\frac{-af^2-4beu^2+4agu^2+4bu^4}{8au^3}.
\end{equation}
For $p,\;q$ defined in this way the equation
$(uT^2+pT+q)^2-V(T)=0$ has exactly one solution
\begin{equation}\label{R15}
T=\frac{-a^2de-e^3+a^3i+(a^2d+3e^2)u^2-3eu^4+u^6}{a^2(-ce+ah-2apq+cu^2)}=:\psi_{1}(u).
\end{equation}
Now, putting $t=\psi_{1}(u)$ we obtain the section
$\sigma_{1}=((u^2-e)/a,\;uT^2+pT+q)$ on the surface
$\cal{E}_{\psi_{1}}$, and similarly to the previous case, we show
that the order of $\sigma_{1}$ is not finite.
\end{proof}

Using similar method we can prove the following

\begin{thm}\label{thm17}

\noindent
\begin{enumerate}
 \item Let $\cal{E}:y^2=x(x^2+f_{2}(t)x+f_{4}(t))$, where $f_{2},\;f_{4}\in\Q[t]$.
 If $\op{deg}f_{2}\leq 2,\;\op{deg}f_{4}\leq 3$ or $\op{deg}f_{2}\leq 2,\;\op{deg}f_{4}=4$
 and at least one of the polynomials $f_{2},\;f_{4}$ is not even,
 then there exists a rational base change $t=\psi(u)$ such that we have a non-torsion section on the
 surface
 $\cal{E}_{\psi}$.
 \item If $\op{deg}f_{2}=2,\;\op{deg}f_{4}=4$ and there is $t_{0}\in\Q$ such that the curve
$\cal{E}_{t_{0}}:y^2=x(x^2+f_{2}(t_{0})x+f_{4}(t_{0}))$ has
infinitely many rational points, then there exists a rational base
change $t=\psi(u)$ such that we have a non-torsion section on the
surface $\cal{E}_{\psi}$ .

\item Let $\cal{E}:y^2=x(x^2+f_{4}(t)x+g_{4}(t))$, where
$f_{4},\;g_{4}\in\Q[t]$ and
 $\op{deg}f_{4}=\op{deg}g_{4}=4$. If at least one of the polynomials $f_{4},\;g_{4}$ is not even,
 then there exists a rational base change $t=\psi(u)$ such that we have a non-torsion section on the
 surface
 $\cal{E}_{\psi}$.
 \end{enumerate}
\end{thm}
\begin{proof}
The proof of part (1) and (2) does not bring any difficulties, and
therefore, it will be omitted (the reasoning is exactly the same
as in the proof of Theorem \ref{thm1}).

We will now outline the proof of part (3) of our theorem. Let
$a,\;b$ be leading coefficients of the polynomials $f_{4},\;g_{4}$
respectively. Let us put $x=b/(u^2-a)$ and treat $\cal{E}$ as a
curve defined over $\Q(u)$. Let us denote this curve by $C_{4}$.
Then, the point at infinity, say $P$, on $C_{4}$ is rational.
Because at least one of the polynomials $f_{4},\;g_{4}$ is not
even, with the use of point $P$ we can construct a non-torsion
section on $\cal{E}$.
\end{proof}

Our previously considered elliptic surfaces (excluding the surface
from part (3) of our Theorem \ref{thm17}) were rational over $\C$.
It means that they are rational over a certain finite extension of
the field $\Q$. It is clear that the questions about such surfaces
can be also asked about general elliptic surfaces.

For $t\in\Q$ let us denote by $\cal{E}_{t}$ the fibre of the
mapping $\pi:\cal{E}\rightarrow \mathbb{P}$ over $t$. In 1992 B.
Mazur proposed an interesting conjecture concerning rational
points on $\cal{E}$.

\begin{conj}\label{conj3}{\rm (Conjecture 4 from \cite{Maz})}

\noindent
 A family of elliptic curves $\{\cal{E}_{t}\}_{t\in\Q}$
fulfills one of the following conditions
\begin{enumerate}
\item for all but finitely many $t\in\Q$ the curve $\cal{E}_{t}$
has the Mordell-Weil rank equal to zero, \item for a dense set of
rational numbers $t$, the Mordell-Weil rank of the curve
$\cal{E}_{t}$ is positive.
\end{enumerate}
\end{conj}

As pointed in \cite{Maz}, the only known example of a family of
elliptic curves fulfilling the condition (1) of the above
conjecture is the constant family with the rank equal to zero, and
it seems quite probable that if the family
$\{\cal{E}_{t}\}_{t\in\Q}$ is non-split, then (1) is not valid.
Examples of families of elliptic curves fulfilling condition (2)
of the above conjecture can be found in
\cite{KuWa},\;\cite{Roh},\;\cite{Man}.

We believe that the following conjecture can be easier to prove

\begin{conj}\label{conj4}
For a non-constant family of elliptic curves
$\{\cal{E}_{t}\}_{t\in\Q}$ there is $t\in\Q$ such that the curve
$\cal{E}_{t}$ has infinitely many rational points.
\end{conj}

As a corollary from the above conjecture we obtain an interesting

\begin{thm}\label{thm18}
Let us assume that Conjecture \ref{conj4} is true. Then for a
non-constant family of elliptic curves $\{\cal{E}_{t}\}_{t\in\Q}$,
the set of rational numbers $t$ such that the rank of
$\cal{E}_{t}$ is positive, is infinite.
\end{thm}
\begin{proof}
Assuming Conjecture \ref{conj4} to be true, we find $t_{1}\in\Q$
such that infinitely many rational points lie on the curve
$\cal{E}_{t_{1}}$. Suppose that we already constructed
$t_{2},\;\ldots,\;t_{n}$ such that the curve $\cal{E}_{t_{i}}$ for
$i=1,\;\ldots,\;n$ has infiniteley many rational points. Let us
further suppose that it is possible to find a polynomial
$h\in\Q[t]$, such that for $i=1,\;\ldots,\;n$ the equation
$h(t)=t_{i}$ does not have a solution in rationals and the system
of equations
\begin{equation}\label{R16}
\begin{cases}
A(t_{1})Y_{1}^4=A(h(T_{1})),\;B(t_{1})Y_{1}^6=B(h(T_{1})),\\ \hskip 2cm \vdots \\
A(t_{n})Y_{n}^4=A(h(T_{n})),\;B(t_{n})Y_{n}^6=B(h(T_{n})).
\end{cases}
\end{equation}
does not have solutions in rationals. From the assumption there
exists $t\in\Q$ such that the curve
$\cal{E}_{h(t)}:\;y^2=x^3+A(h(t))x+B(h(t))$ has a positive rank.
Defining now $t_{n+1}=h(t)$ and repeating the reasoning, we obtain
the statement of our theorem.

We will show now that there exists a polynomial $h\in\Q[t]$
fulfilling the above conditions. Let $h_{1}$ be polynomial in
$\Q[t]$ such that the equation $h_{1}(t)=t_{i}$ for
$i=1,\;\ldots,\;n$ does not have any solutions. Clearly it is
enough to show the existence of our polynomial for the first row
in the system of equations (\ref{R16}). Let us, therefore,
consider the system of equations
$A(t_{1})Y_{1}^4=A(h_{1}(T_{1})),\;B(t_{1})Y_{1}^6=B(h_{1}(T_{1}))$.
If $A(t_{1})B(t_{1})=0$, then this system has finitely many
rational solutions, and we will find $h_{2}\in\Q[t]$ such that the
polynomial $h=h_{1}\circ h_{2}$ fulfills the desired conditions.
Thus, let us assume that $A(t_{1})B(t_{1})\neq 0$. If
$(A(h_{1}(T_{1}))/A(t_{1}))^3\neq (B(h_{1}(T_{1}))/B(t_{1}))^2$,
then our system has at most $3\op{deg}(A\circ h_{1}) +
2\op{deg}(B\circ h_{1})$ solutions in $\Q$ and there exists
$h_{2}\in\mathbb{Q}[t]$ such that the polynomial $h=h_{1}\circ
h_{2}$ fulfills the desired conditions. In case when
$(A(h_{1}(T_{1}))/A(t_{1}))^3= (B(h_{1}(T_{1}))/B(t_{1}))^2$, the
problem is reduced to the examination of the curve
$C:\;Y_{1}^2=H(T_{1})$, where $H$ is a polynomial such that
$A(h_{1}(T_{1}))/A(t_{1})=H(T_{1})^2,\;B(h_{1}(T_{1}))/B(t_{1})=H(T_{1})^3$.
If the polynomial $H$ was a square of another polynomial, we would
obtain that the family $\{\mathcal{E}_{t}\}_{t\in\mathbb{Q}}$ is
constant, which contradicts the assumption. Therefore, we see that
the polynomial $H$ is not a square. Now, there exists a polynomial
$h_{2}$, such that the genus of the curve
$C':\;Y_{1}^2=H(h_{2}(T_{1}))$ is $\geq 2$. From the Faltings
theorem there are only finitely many rational points on $C'$; so
after a polynomial change of variable, we obtain a polynomial
fulfilling all the necessary conditions. Applying this reasoning
to the second,$\;\ldots,\;n$-th equation in the system (\ref{R16})
we obtain the statement of our theorem.
\end{proof}

In view of above theorem a natural question arises

\begin{ques}\label{ques3}
Are conditions (2) of Conjecture \ref{conj3} and Conjecture
\ref{conj4} equivalent?
\end{ques}


\bigskip

 \hskip 4.5cm       Maciej Ulas

 \hskip 4.5cm       Jagiellonian University

 \hskip 4.5cm       Institute of Mathematics

 \hskip 4.5cm       Reymonta 4

 \hskip 4.5cm       30 - 059 Krak\'{o}w, Poland

 \hskip 4.5cm      e-mail:\;{\tt Maciej.Ulas@im.uj.edu.pl}
\end{document}